\documentclass[11pt, oneside]{amsart}
\usepackage[text={5.58in,8.5in},centering,letterpaper,dvips]{geometry}
\usepackage[dvipsnames]{xcolor}
\usepackage{graphicx}
\usepackage{subcaption}
\usepackage{amsfonts}
\usepackage{epsf}
\usepackage{amssymb}
\usepackage{amsmath}
\usepackage{amscd}
\usepackage{tikz}
\usepackage{fancyhdr}
\usepackage{setspace}
\usepackage{hyperref}
\usepackage[all]{xy}
\usetikzlibrary{matrix}
\usepackage{verbatim}
\usepackage{enumerate}
\usepackage{nicematrix}
\usepackage{blkarray}

\theoremstyle{theorem}
\newtheorem{theorem}{Theorem}[section]
\newtheorem{proposition}[theorem]{Proposition}
\newtheorem{lemma}[theorem]{Lemma}
\newtheorem{question}[theorem]{Question}
\newtheorem{corollary}[theorem]{Corollary}

\theoremstyle{definition}

\newtheorem{remark}[theorem]{Remark}

\newcommand{\Z}{\mathbb{Z}}

\newcommand{\N}{\mathbb{N}}

\newcommand{\X}{\times}

\newcommand{\col}{\text{col}}

\newcommand{\va}{\vec{\mathbf{a}}}
\newcommand{\vb}{\vec{\mathbf{b}}}
\newcommand{\vs}{\vec{\mathbf{s}}}

\newcommand{\ve}{\vec{\mathbf{e}}}
\newcommand{\vx}{\vec{\mathbf{x}}}
\newcommand{\vz}{\vec{\mathbf{0}}}

\newcommand{\nullity}{\text{nullity}}

\definecolor{dgray}{gray}{0.35}
\definecolor{PURP}{RGB}{.6,.2,.8}

\makeatletter
\def\@seccntformat#1{%
  \protect\textup{\protect\@secnumfont
    \ifnum\pdfstrcmp{subsection}{#1}=0 \bfseries\fi
    \csname the#1\endcsname
    \protect\@secnumpunct
  }%
}  
\makeatother

\makeatletter
\newtheorem*{rep@theorem}{\rep@title}
\newcommand{\newreptheorem}[2]{%
\newenvironment{rep#1}[1]{%
 \def\rep@title{#2 \ref{##1}}%
 \begin{rep@theorem}}%
 {\end{rep@theorem}}}
\makeatother

\newreptheorem{theorem}{Theorem}
\newreptheorem{lemma}{Lemma}
\newreptheorem{question}{Question}
\newreptheorem{corollary}{Corollary}
\newreptheorem{proposition}{Proposition}

\begin{document}

\rhead{\thepage}
\lhead{\author}
\thispagestyle{empty}


\raggedbottom
\pagenumbering{arabic}
\setcounter{section}{0}


\title{Colorings of symmetric unions and partial knots}

\author{Ben Clingenpeel} %
\address{Department of Mathematics, The George Washington University, Washington DC, USA}
\email{ben.clingenpeel@gwmail.gwu.edu}

\author{Zongzheng Dai} %
\address{Department of Mathematics, Rensselaer Polytechnic Institute, Troy, NY, USA}
\email{daiz4@rpi.edu}

\author{Gabriel Diraviam} %
\address{}
\email{gabe.math01@gmail.com}

\author{Kareem Jaber} %
\address{Department of Mathematics, Princeton University, Princeton, NJ, USA}
\email{kj5388@princeton.edu}

\author{Krishnendu Kar} %
\address{Department of Mathematics, Louisiana State University, Baton Rouge, LA, USA}
\email{kkar2@lsu.edu}

\author{Ziyun Liu} %
\address{Department of Mathematical Sciences, Carnegie Mellon University, Pittsburgh, PA, USA}
\email{ziyunliu@andrew.cmu.edu}

\author{Teo Miklethun} %
\address{}
\email{teomiklethun@gmail.com}

\author{Haritha Nagampoozhy} %
\address{Indian Institute of Science Education and Research, Pune, India}
\email{harithnag@gmail.com}

\author{Michael Perry} %
\address{Department of Mathematics, Syracuse University, Syrcause, NY, USA}
\email{mperry03@syr.edu}

\author{Moses Samuelson-Lynn} %
\address{Max Planck Institute for Mathematics in the Sciences, Leipzig, Germany}
\email{m.samuelsonlynn@utah.edu}

\author{Eli Seamans} %
\address{Department of Mathematics, Syracuse University, Syrcause, NY, USA}
\email{ewseaman@syr.edu}

\author{Ana Wright} %
\address{Department of Mathematics and Computer Science, Davidson College, Davidson, NC, USA}
\email{anwright1@davidson.edu}

\author{Nicole (Xinyu) Xie}
\address{Department of Mathematics, University of Nebraska-Lincoln, Lincoln, NE, USA}
\email{nxie2@huskers.unl.edu}

\author{Ruiqi Zou} %
\address{Department of Mathematics, University of Virginia, Charlottesville, VA, USA}
\email{nbc3hn@virginia.edu}

\author{Alexander Zupan} %
\address{Department of Mathematics, University of Nebraska-Lincoln, Lincoln, NE, USA}
\email{zupan@unl.edu}
\urladdr{https://math.unl.edu/person/alex-zupan/}

\begin{abstract}
Motivated by work of Kinoshita and Teraska, Lamm introduced the notion of a \emph{symmetric union}, which can be constructed from a \emph{partial knot} $J$ by introducing additional crossings to a diagram of $J \# -\!J$ along its axis of symmetry.   If both $J$ and $J'$ are partial knots for different symmetric union presentations of the same ribbon knot $K$, the knots $J$ and $J'$ are said to be \emph{symmetrically related}.  Lamm proved that if $J$ and $J'$ are symmetrically related, then $\det J = \det J'$, asking whether the converse is true.  In this article, we give a negative answer to Lamm's question, constructing for any natural number $m$ a family of $2^m$ knots with the same determinant but such that no two knots in the family are symmetrically related.  This result is a corollary to our main theorem, that if $J$ is the partial knot in a symmetric union presentation for $K$, then $\text{col}_p(J) \leq \text{col}_p(K) \leq \frac{(\text{col}_p(J))^2}{2}$, where $\col_p(\cdot )$ denotes the number of $p$-colorings of a knot.
\end{abstract}

\maketitle

\section{Introduction}

The slice-ribbon problem of Fox has been a significant motivator in low-dimensional topology for over 60 years~\cite{fox}:  Every knot $K \subset S^3$ that bounds an immersed disk with ribbon intersections, a \emph{ribbon disk}, can be seen to bound a smoothly embedded disk in $B^4$, a \emph{slice disk}.  In other words, every ribbon knot is slice, but it remains open whether the converse is true.  Building on work of Kinoshita and Terasaka~\cite{kt}, Lamm introduced the idea of a symmetric union, a construction obtained by starting with a knot $J$, called a \emph{partial knot}, and adding crossings to a diagram of the ribbon knot $J \# -\!J$ (where $-J$ denotes the mirror image of $J$) along a decomposing sphere~\cite{lamm1}.  Analogous to what happens in the slice-ribbon problem, it is straightforward to show that if $K$ admits a symmetric union presentation, then $K$ is a ribbon knot, and we call $K$ a \emph{symmetric ribbon knot}.  However, the following is open:

\begin{question}\cite{lamm1}\label{q1}
Is every ribbon knot a symmetric ribbon knot?
\end{question}

In his introductory paper on the subject~\cite{lamm1}, Lamm established a number of results relating the invariants of a symmetric ribbon knot $K$ with those of a corresponding partial knot $J$.  In particular, $\det(K) = (\det(J))^2$.  In another paper answering Question~\ref{q1} in the affirmative for two-bridge knots $K$, Lamm asked a different question:

\begin{question}\cite{lamm2}\label{q2}
Suppose $\det(J) = \det(J')$.  Must there exist a knot $K$ admitting symmetric union presentations with partial knots $J$ and $J'$?
\end{question}

If such a $K$ exists, we say that $J$ and $J'$ are \emph{symmetrically related}.  In the work at present, we compare the number of admissible $p$-colorings $\col_p(K)$ and $\col_p(J)$ of a symmetric ribbon knot $K$ and its partial knot, respectively.  We prove

\begin{theorem}\label{thm:main}
Suppose $K$ admits a symmetric union presentation with partial knot $J$.  Then
\[ \col_p(J) \leq \col_p(K) \leq \frac{(\col_p(J))^2}{p}.\]
\end{theorem}
As a corollary, we obtain a negative answer to Question~\ref{q2}.

\begin{corollary}\label{cor:main}
For any $m \in \N$, there exists a collection $\mathcal{K}$ of $2^m$ knots with the same determinant but such that for any $K,K' \in \mathcal{K}$, we have that $K$ and $K'$ are not symmetrically related.
\end{corollary}

\begin{remark}
There exist examples of a symmetric ribbon knot $K$ with a partial knot $J$ such that $\col_p(J) = \col_p(K)$, and there are also examples of $K$ and $J$ such that $\col_p(K) = \frac{(\col_p(J))^2}{p}$.  Thus, it is possible to attain equality at either end of the range given in Theorem~\ref{thm:main}.  See Remark~\ref{rmk:range} below for more details.
\end{remark}

\begin{remark}
While Question~\ref{q2} is relatively unexplored, Question~\ref{q1} has received considerable attention in the literature.  For example, in~\cite{aceto}, Aceto showed that the difference between ribbon number and symmetric ribbon number can be arbitrarily large.  In~\cite{lamm3}, Lamm searched for and catalogued potential ribbon knots up to 12 crossings for which no symmetric union presentation has been found.  Finally, a recent paper of Boileau, Kitano, and Nozaki proved that if the ribbon knot $11a_{201}$ identified in Lamm's search does indeed admit a symmetric union presentation, then the corresponding partial knot must be either $6_1$ or $9_1$~\cite{BKN}.
\end{remark}

The plan of the article is as follows:  In Section~\ref{sec:prelim}, we establish the necessary preliminaries and discuss a strange example, showing that the proof of Theorem~\ref{thm:main} requires more than expected.  In Section~\ref{sec:proof}, we prove Theorem~\ref{thm:main} and Corollary~\ref{cor:main}.  Finally, in Section~\ref{sec:quest}, we include open questions for future consideration. \\

\noindent \textbf{Acknowledgements:}  We express our deep gratitude to the 2022 Polymath Jr. REU, during which this work was completed, for inducing a supportive, collaborative research environment.  We are grateful to Michel Boileau and Christoph Lamm for helpful comments on a draft of this work.  Finally, we thank Michel Boileau for mentioning this problem during his talk~\cite{boileau} as part of the First International On-line Knot Theory Congress, which gave us the push we needed to finish this project.  AZ was supported by NSF awards DMS-2005518 and DMS-2405301 and a Simons Foundation Travel Award.

\section{Preliminaries}\label{sec:prelim}

First, we formally define symmetric unions, introduced by Lamm~\cite{lamm1} and motivated by work of Kinoshita and Teraska~\cite{kt}.  Let $D$ be a knot diagram, and consider the reflection $-D$ of $D$ over a vertical axis $\ell$ disjoint from $D$.  A symmetric union is obtained by performing specific tangle replacements on the diagram $D \sqcup -D$ as follows:  Consider $k+1$ tangles determined by disks $\Delta_0,\dots,\Delta_k$ in the projection plane, where $\Delta_i$ has reflection symmetry over $\ell$ and meets each of $D$ and $-D$ in a single crossing-less strand.  For integers $n_1,\dots,n_k$, let $D \sqcup -D(n_1,\dots,n_k)$ denote the diagram obtained by replacing the tangle in $\Delta_0$ with the $\infty$-tangle and by replacing the tangle in $\Delta_i$ with a tangle diagram consisting of $n_i$ (signed) crossings contained in the axis of symmetry $\ell$.  A schematic diagram is shown in Figure~\ref{fig:sup}.

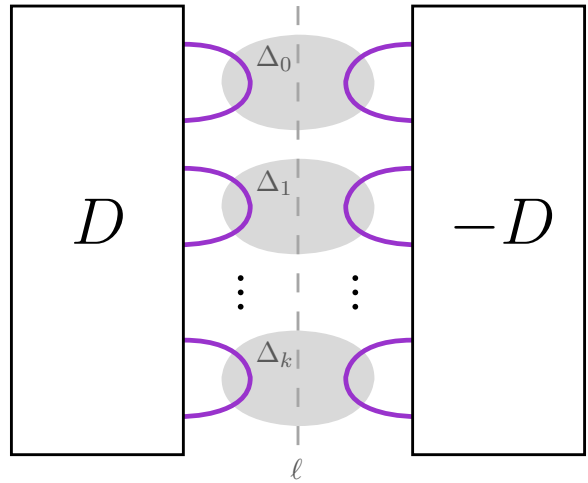
\begin{figure}[h!]
  \centering
  \includegraphics[width=.6\linewidth]{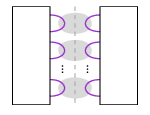}
  \put (-133,9) {\textcolor{dgray}{$\ell$}}
    \put (-146,118) {\textcolor{dgray}{$\Delta_1$}}
    \put (-146,165) {\textcolor{dgray}{$\Delta_0$}}
    \put (-146,53) {\textcolor{dgray}{$\Delta_k$}}
    \put (-216,98) {\Huge{$D$}}
    \put (-73,98) {\Huge{$-D$}}
	\caption{The general form used in the construction of a symmetric union.}
	\label{fig:sup}
\end{figure}

A \emph{symmetric union presentation} for a knot $K$ is a realization of $K$ as a symmetric union of the form $D \sqcup -D(n_1,\dots,n_k)$.  The knot type $J$ of the diagram $D$ is called a \emph{partial knot} for $K$.  As an example, in Figure~\ref{fig:steve} we see symmetric union presentations $D \sqcup -D(1)$ and $D \sqcup -D(3)$ for the knots $K = 6_1$ and $K = 9_{46}$, respectively, with partial knot $J$ the trefoil $3_1$.

\begin{figure}[h!]
  \centering
  \includegraphics[width=.45\linewidth]{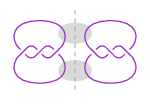} \\
      \includegraphics[width=.4\linewidth]{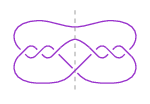}
  \qquad
    \includegraphics[width=.4\linewidth]{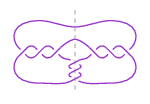}
	\caption{At top, an initial setup to construct a symmetric union with partial knot the trefoil.  At bottom left, a symmetric union $D \sqcup -D(1)$ for the knot $K = 6_1$.  At bottom right, a symmetric union $D \sqcup -D(3)$ for the knot $K = 9_{46}$.}
	\label{fig:steve}
\end{figure}

In a well-known construction, every knot $K$ that admits a symmetric union presentation is a ribbon knot.  As mentioned above in Question~\ref{q1}, at the present it is unknown whether every ribbon knot admits a symmetric union presentation.

Given a diagram $D$, a \emph{$p$-coloring} of $D$ is an assignment of an element of $\Z_p$ to each connected strand of $D$ with the rule that at each crossing, if the overstrand is labeled $a_i$ and the two understrands are labeled $a_j$ and $a_k$, then
\[ a_j + a_k \equiv 2 a_i \, \text{mod } p \quad \text{ or, equivalently, } \quad  a_j + a_k - 2 a_i \equiv 0 \, \text{mod } p.\]
Let $\text{col}_p(D)$ denote the number of $p$-colorings of a diagram $D$, where we count both \emph{trivial} $p$-colorings (in which every strand has the same label) and \emph{non-trivial} $p$ colorings.  It is well-known that $\text{col}_p(D)$ is preserved under Reidemeister moves, and thus, we can define the knot invariant $\text{col}_p(K)$ to be $\text{col}_p(D)$ for any diagram $D$ corresponding to $K$.  Because every diagram admits $p$ trivial colorings, we have $\col_p(K) \geq p$ for all knots $K$.  When $\col_p(K) > p$ (that is, when $K$ admits a non-trivial $p$-coloring), we say that $K$ is \emph{$p$-colorable}.

See, for example, \cite{prz} for further details and background, including the following proposition, which relates $p$-colorings of a connected sum to $p$-colorings of its summands.

\begin{proposition}\cite{prz}\label{prop:sum}
Let $K$ and $K'$ be any knots.  Then
\[ \text{col}_p(K \# K') = \frac{\col_p(K) \col_p(K')}{p}.\]
\end{proposition}

As an immediate corollary, we have

\begin{corollary}\label{cor:notcolor}
Let $K$ and $K'$ be any knots such that $K'$ is not $p$-colorable.  Then
\[ \text{col}_p(K \# K') = \col_p(K).\]
\end{corollary}

We will also require the following well-known result relating $p$-colorability and knot determinants.  For a proof, see Chapter 3 of~\cite{livingston}, for instance.

\begin{lemma}\label{lem:detcol}
A knot $K$ is $p$-colorable if and only if $p$ divides $\det(K)$.
\end{lemma}

We need one more lemma as input.  A knot $K \subset S^3$ is \emph{2-bridge} if there exists a 2-sphere $\Sigma$ splitting $K$ into two trivial 2-strand tangles, $(S^3,K) = (B^3,\tau_1) \cup_{\Sigma} (B^3,\tau_2)$, in which the strands of $\tau_i$ are simultaneously isotopic rel boundary into $\Sigma$.  For example, any $(p,2)$-torus knot is a 2-bridge knot.  Przytycki proved

\begin{lemma}\label{lem:bridge}\cite[Corollary 1.7]{prz}
If $K$ is a 2-bridge knot, then $\col_p(K) \leq p^2$.
\end{lemma}

Now, we consider $p$-colorings of symmetric union presentations.  Suppose that $K$ admits a symmetric union presentation with partial knot $J$, and suppose that a diagram $D$ for $J$ is $p$-colored.  Then we can $p$-color the mirror image $-D$ identically, so that for any disk $\Delta_j$ as labeled above, both strands of $D\sqcup -D$ meeting $\Delta_j$ have the same color.  Replacing these strands with an $\infty$-tangle or some number of crossings, all of which have the same color, produces an admissible coloring of $D\sqcup -D(n_1,\dots,n_k)$, and so we can see that
\[ \col_p(J) \leq \col_p(K),\]
establishing one of the inequalities in Theorem~\ref{thm:main}.  See Figure~\ref{fig:color1} for an example in which $p=3$, in which case we can use three actual colors, and any assignment of labels in $\Z_3$ to the three colors suffices.

\begin{figure}[h!]
  \centering
   \includegraphics[width=.2\linewidth]{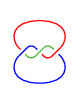} \qquad
 \includegraphics[width=.4\linewidth]{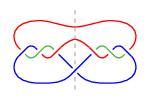}
	\caption{A 3-coloring of the partial knot $3_1$ (left) induces a 3-coloring of a symmetric union presentation of $6_1$ (right).}
	\label{fig:color1}
\end{figure}

The other inequality in Theorem~\ref{thm:main} is significantly more subtle:  Suppose $K$ has a $p$-colored symmetric union presentation $D\sqcup -D(n_1,\dots,n_k)$ such that for $1 \leq j \leq k$, both strands exiting $\Delta_j$ on the left are colored the same, and both strands exiting $\Delta_j$ on the right are colored the same.  We call such a $p$-coloring \emph{symmetrically compatible}.  In this case, the symmetrically compatible $p$-coloring induces a $p$-coloring of $J \# -J$ by removing the crossings in $\Delta_j$ and reconnecting the strands on the left and right.  See Figure~\ref{fig:color2} for an example in which $p=3$.

\begin{figure}[h!]
  \centering
   \includegraphics[width=.4\linewidth]{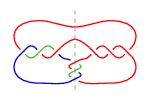} \qquad
 \includegraphics[width=.4\linewidth]{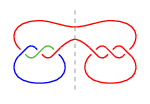}
	\caption{A symmetrically compatible 3-coloring of $9_{46}$ (left) induces a 3-coloring of $3_1 \# 3_1^*$ (right).}
	\label{fig:color2}
\end{figure}

Prytyckzi's proof of Proposition~\ref{prop:sum} reveals that every admissible $p$-coloring of the diagram $D \# -\!D$ is symmetrically compatible, which in this case means that the two strands crossing the axis of symmetry are colored the same.  One might hope that a similar statement could hold for symmetric union presentations.  If every $p$-coloring of the symmetric union $D\sqcup -D(n_1,\dots,n_k)$ were symmetrically compatible, we would get an injection from the set of admissible $p$-colorings of $K$ into the set of admissible $p$-colorings of $J \# -\!J$, which would imply $\col_p(K) \leq \col_p(J \# -\!J) = \frac{(\col_p(J))^2}{p}$, our desired inequality, by Proposition~\ref{prop:sum}.  However, this is not the case, as noted below.

\begin{proposition}\label{prop:notsymm}
There exists a knot $K$ with a 3-colorable symmetric union presentation that is not symmetrically compatible.
\end{proposition}

\begin{proof}
In Figure~\ref{fig:incompatible}, we include a 3-colored symmetric union presentation for $K = 3_1 \# -\!3_1$ with partial knot $3_1$ that is not symmetrically compatible.
\end{proof}

\begin{figure}[h!]
  \centering
   \includegraphics[width=.4\linewidth]{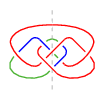}
	\caption{A 3-coloring of $3_1 \# -\!3_1$ that is not symmetrically compatible.}
	\label{fig:incompatible}
\end{figure}

We conclude that this naive approach will not work to prove Theorem~\ref{thm:main}, and so we turn to coloring matrices, defined and discussed below.

To compute $\text{col}_p(D)$ for a given diagram $D$, we define the $r \X r$ \emph{coloring matrix} $M_D$ with entries in $\Z_p$ as follows:  Enumerate the crossings $c_1,\dots,c_r$ of $D$ and the strands $s_1,\dots,s_r$ (noting that the number of crossings is equal to the number of strands).  The $ij$-th entry of $M_D$ is 1 if the crossing $c_i$ includes the strand $s_j$ as an understrand, $-2$ if $c_i$ includes $s_j$ as the overstrand, and 0 otherwise.  A vector $\vx \in \Z_p^n$ represents a $p$-coloring if and only if $M_D \cdot \vx = \vz$ (working modulo $p$).  Thus, letting $\text{nullity}_p(M_D)$ denote the mod $p$ nullity of $M_D$, we have
\[ \text{col}_p(D) = (\text{nullity}_p(M_D))^p.\]
It follows that despite the fact that a coloring matrix $M_D$ depends on the choice of diagram $D$, its mod $p$ nullity does not, yielding an invariant of a knot $K$.  Note that these conventions differ slightly from the \emph{mod $p$ rank} of a knot $K$ discussed in Chapter 3 of~\cite{livingston}, in which a row and column of $M_D$ are deleted.

\section{Proof of the main theorem}\label{sec:proof}

In this section, we analyze the coloring matrices corresponding to symmetric union presentations in order to prove Theorem~\ref{thm:main}.  We also prove Corollary~\ref{cor:main}, showing that there exist arbitrarily large collections of knots with the same determinant such that no two are symmetrically related.  Our main tool is the proposition below, which allows us to replace the coloring matrix coming from a symmetric union presentation with another matrix having the same mod $p$ nullity.

\begin{proposition}\label{prop:matrix}
Let $K$ be a symmetric ribbon knot with partial knot $J$.  Then there is a symmetric union presentation $D \sqcup -D(n_1,\dots,n_k)$ for $K$, with $D$ a diagram for $J$, such that the coloring matrix for $D \sqcup -D(n_1,\dots,n_k)$ has the same mod $p$ nullity as some matrix of the form

\[\left[ 
\begin{array}{c|c} 
  M_D & A \\ 
  \hline 
0 & M_D' 
\end{array} 
\right], \]

\noindent where the first column of the matrix $A$ is $\ve_1$, and $M_D'$ is obtained from $M_D$ by replacing its first column with the zero vector.
\end{proposition}

\begin{proof}
Suppose $K$ has a symmetric union presentation $D \sqcup -D(n_1,\dots,n_k)$.  After performing Reidemeister 1 moves on $D$ if necessary (and at the expense of increasing $k$), we can assume that $n_i = \pm 1$ for all $i$.  Letting $\Delta_0,\dots,\Delta_k$ be the disks described in Section~\ref{sec:prelim}, we can perform additional Reidemeister moves to ensure that each strand $s_i$ of $D$ meets at most one disk $\Delta_j$, chosen so that the strand $s_1$ (after relabeling if necessary) passing through $\Delta_0$ terminates immediately outside of $\Delta_0$ and without passing through any overcrossings.  In an abuse of notation, we let $D \sqcup -D(n_1,\dots,n_k)$ denote the new symmetric union presentation.  The setup is shown in Figure~\ref{fig:R1}.

\begin{figure}[h!]
  \centering
  \includegraphics[width=.44\linewidth]{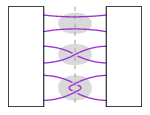} \quad \raisebox{2.5cm}{$\longrightarrow$} \quad
    \includegraphics[width=.44\linewidth]{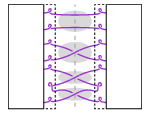}
	\caption{After performing R1 moves on the symmetric union presentation at left, each parameter $n_i$ at right is $\pm 1$, and each strand intersects at most one disk $\Delta_j$ at most once.}
	\label{fig:R1}
\end{figure}

Let $\vs_1,\dots,\vs_r$ be the column vectors of the coloring matrix $M_D$.  By construction, $\vs_1$ corresponds to the strand passing through $\Delta_0$ and has two nonzero entries of $+1$.  Every other column vector $\vs_i$ has two nonzero entries of $+1$ and some number (possibly zero) of entries of $-2$ corresponding to the overcrossings contained in the strand $s_i$.  Now, we break each strand $s_i$ passing through some $\Delta_j$ into the union of two strands $a_i$ and $b_i$, where $a_i$ exits $\Delta_j$ above $b_i$, as shown in at left in Figure~\ref{fig:cut}.  In addition, we express $\vs_i$ as the sum of vectors $\va_i$ and $\vb_i$, where $\va_i$ records the crossing information from the strand $a_i$ and $\vb_i$ records the crossing information from the strand $b_i$.

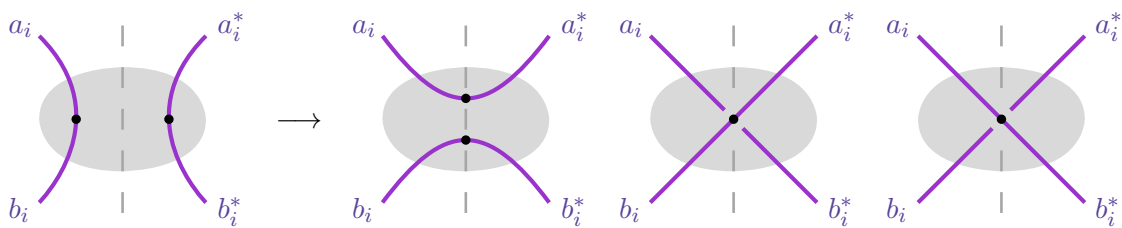
\begin{figure}[h!]
  \centering
  \includegraphics[width=.2\linewidth]{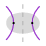} \quad \raisebox{1.4cm}{$\longrightarrow$}\quad
    \includegraphics[width=.2\linewidth]{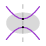} \quad
      \includegraphics[width=.2\linewidth]{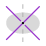} \quad
        \includegraphics[width=.2\linewidth]{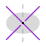}
          \put (-419,6) {\textcolor{Violet}{$b_i$}}
                   \put (-419,76) {\textcolor{Violet}{$a_i$}}
          \put (-340,6) {\textcolor{Violet}{$b_i^*$}}
                   \put (-340,76) {\textcolor{Violet}{$a_i^*$}}
                             \put (-289,6) {\textcolor{Violet}{$b_i$}}
                   \put (-289,76) {\textcolor{Violet}{$a_i$}}
          \put (-210,6) {\textcolor{Violet}{$b_i^*$}}
                   \put (-210,76) {\textcolor{Violet}{$a_i^*$}}
                                                \put (-188,6) {\textcolor{Violet}{$b_i$}}
                   \put (-188,76) {\textcolor{Violet}{$a_i$}}
          \put (-109,6) {\textcolor{Violet}{$b_i^*$}}
                   \put (-109,76) {\textcolor{Violet}{$a_i^*$}}
                                                                   \put (-87,6) {\textcolor{Violet}{$b_i$}}
                   \put (-87,76) {\textcolor{Violet}{$a_i$}}
          \put (-8,6) {\textcolor{Violet}{$b_i^*$}}
                   \put (-8,76) {\textcolor{Violet}{$a_i^*$}}
	\caption{Cutting and reassembling strands in a a symmetric union presentation}
	\label{fig:cut}
\end{figure}

Let $-D$ be the mirror image of $D$, with strands $s_i^*$ broken into $a_i^*$ and $b_i^*$ and corresponding vectors $\vs_i^* = \va_i^* + \vb_i^*$.  Note that $\vs_i^* = \vs_i$, $\va_i^* = \va_i$, and $\vb_i^* = \vb_i$, and thus $M_{-D} = M_D$.  The notation may seem redundant, but it will be important for us to keep track of the information coming from $D$ and the information coming from $-D$ as we show how to modify a coloring matrix for $D \sqcup -D$ to get a coloring matrix for $D \sqcup -D(n_1,\dots,n_k)$.  To begin, observe that
\[ M_{D \sqcup -D} = 
\left[ \begin{array}{c|c} 
  M_D & 0 \\ 
  \hline 
0 & M_D
\end{array} 
\right]. \]

Replacing the tangle in $\Delta_0$ with the $\infty$-tangle connects strands $a_i$ and $a_i^*$ and $b_i$ and $b_i^*$.  The effect on $M_{D \sqcup -D}$ is

\[
    \begin{bmatrix}
    \,\,\begin{matrix}
        \vs_1 \\
        \hline
       \vec{0} 
    \end{matrix}\,\,\,\,
    \cdots
    \,\,\,\,\begin{matrix}
        \vline & \vec{0} \\
        \hline
        \vline & \vs_1^* 
    \end{matrix}\,\,\,\,
    \cdots
    \end{bmatrix}
    =
    \begin{bmatrix}
    \,\,\begin{matrix}
         \va_1 + \vb_1 \\
        \hline
        \vec{0}
    \end{matrix}\,\,\,\,
    \cdots
    \,\,\,\,\begin{matrix}
        \vline & \vec{0} \\
        \hline
        \vline & \va_1^*+ \vb_1^*
    \end{matrix}\,\,\,\,
    \cdots
    \end{bmatrix}
    \\\to
    \begin{bmatrix}
    \,\,\begin{matrix}
         \va_1\\
        \hline
      \va_1^* 
    \end{matrix}\,\,\,\,
    \cdots
    \,\,\,\,\begin{matrix}
        \vline & \vb_1 \\
        \hline
        \vline & \vb_1^* 
    \end{matrix}\,\,\,\,
    \cdots
    \end{bmatrix}
\]

In other words, viewing the columns of $M_{D \sqcup -D}$ as vectors in $\Z_p^r \X \Z_p^r$, the first column $(\vs_1,\vec 0)$ is replaced with $(\va_1,\va_1^*)$, and column $r+1$, initially $(\vec 0,\vs_1^*)$, is replaced with $(\vb_1,\vb_1^*)$.

Next, we examine the change induced by adding a $+1$-crossing to the strands $s_i$ and $s_i^*$ along the axis of symmetry $\ell$ within the disk $\Delta_j$.  This addition adds a crossing and converts two strands into three strands, which will append a new row and column to our coloring matrix.  We set the convention that $s_i$ is replaced with $a_i$, $s_i^*$ is replaced with $b_i^*$, and the new strand, $t_j$ is obtained by joining $b_i$ to $a_i^*$, as shown in Figure~\ref{fig:cut}.  The new coloring matrix has an extra row with three non-zero entries, an entry of $+1$ in the columns corresponding to the replacements of $s_i$ and $s_i^*$ and an entry of $-2$ in the column corresponding to $t_j$.  This change is depicted below.

\begin{align*}
    \begin{bmatrix}
    \cdots
    \,\,\,\,\begin{matrix}
     \va_i + \vb_i \\
        \hline
      \vec{0}
    \end{matrix}\,\,\,\,
    \cdots
    \,\,\,\,\begin{matrix}
     \vline\\
        \hline
 \vline
    \end{matrix}\,\,\,\,
    \cdots
    \,\,\,\,\begin{matrix}
        \vec{0} \\
        \hline
       \va_i^* + \vb_i^* 
    \end{matrix}\,\,\,\,
    \cdots
    \end{bmatrix}
    \to&
    \begin{bmatrix}
\cdots
    \,\,\,\,\begin{matrix}
     \va_i \\
        \hline
      \vec{0} \\
      \hline
      1
    \end{matrix}\,\,\,\,
    \cdots
    \,\,\,\,\begin{matrix}
     \vline\\
        \hline
 \vline \\
 \hline \\
    \end{matrix}\,\,\,\,
    \cdots
    \,\,\,\,\begin{matrix}
        0 \\
        \hline
       \vb_i^* \\
       \hline
       1
    \end{matrix}\,\,\,\,
    \cdots
    \,\,\,\,
    \begin{matrix}
        \vline & \vb_i \\
        \hline
        \vline & \va_i^* \\
        \hline
         & -2
    \end{matrix}
    \end{bmatrix}.
\end{align*}

Similarly, we can examine the change induced by adding a $-1$-crossing to the strands $s_i$ and $s_i^*$ along the axis of symmetry $\ell$ within the disk $\Delta_j$.  As above, we append a new row and column to our coloring matrix.  We replace $s_i$ with $b_i$, strand $s_i^*$ is replaced with $a_i^*$, and the new strand, $t_j$ is obtained by joining $a_i$ to $b_i^*$, also shown in Figure~\ref{fig:cut}.  The new coloring matrix has an extra row with three non-zero entries, an entry of $+1$ in the columns corresponding to the replacements of $s_i$ and $s_i^*$ and an entry of $-2$ in the column corresponding to $t_j$, as shown below.

\begin{align*}
    \begin{bmatrix}
    \cdots
    \,\,\,\,\begin{matrix}
     \va_i + \vb_i \\
        \hline
      \vec{0}
    \end{matrix}\,\,\,\,
    \cdots
    \,\,\,\,\begin{matrix}
     \vline\\
        \hline
 \vline
    \end{matrix}\,\,\,\,
    \cdots
    \,\,\,\,\begin{matrix}
        \vec{0} \\
        \hline
       \va_i^* + \vb_i^* 
    \end{matrix}\,\,\,\,
    \cdots
    \end{bmatrix}
    \to&
    \begin{bmatrix}
\cdots
    \,\,\,\,\begin{matrix}
     \vb_i \\
        \hline
      \vec{0} \\
      \hline
      1
    \end{matrix}\,\,\,\,
    \cdots
    \,\,\,\,\begin{matrix}
     \vline\\
        \hline
 \vline \\
 \hline \\
    \end{matrix}\,\,\,\,
    \cdots
    \,\,\,\,\begin{matrix}
        0 \\
        \hline
       \va_i^* \\
       \hline
       1
    \end{matrix}\,\,\,\,
    \cdots
    \,\,\,\,
    \begin{matrix}
        \vline & \va_i \\
        \hline
        \vline & \vb_i^* \\
        \hline
         & -2
    \end{matrix}
    \end{bmatrix}.
\end{align*}

After making these changes for every disk $\Delta_1,\dots,\Delta_k$, we obtain a $(2r+k,2r+k)$-coloring matrix $M_1$ corresponding to $D \sqcup -D(n_1,\dots,n_k)$.  Since row and column operations preserve the nullity of a matrix, we can perform these operations freely.  We will express the columns of $M_1$ as vectors in $\Z_p^r \X \Z_p^r \X \Z_p^k$ to reflect the block structure above, and we drop the $*$ notation, recalling that $\va_i^* = \va_i$ and $\vb_i^* = \vb_i$.  To summarize the discussion above, if a disk $\Delta_j$ contains a $+1$-crossing, then we convert columns $i$ and $r+i$ of $M_{D \sqcup D}$, which are $(\va_i+\vb_i,\vec 0)$ and $(\vec 0,\va_i+\vb_i)$, respectively, into columns $i$, $r+i$, and $2r+j$, $(\va_i,\vec 0, \ve_j)$, $(\vec 0, \vb_i, \ve_j)$, and $(\vb_i,\va_i, -2\ve_j)$.  A sequence of column operations replaces the three vectors with $(\va_i,-\vb_i, \vec 0)$, $(\vec 0, \vb_i, \ve_j)$, and $(\va_i + \vb_i,\va_i + \vb_i, \vec 0)$, which we permute to be in columns $i$, $2r+j$, and $r+i$, respectively.

On the other hand, if $\Delta_j$ contains a $-1$-crossing, the corresponding induced columns are $(\vb_i,\vec 0, \ve_j)$, $(\vec 0, \va_i, \ve_j)$, and $(\va_i, \vb_i, -2\ve_j)$.  Column operations send these vectors to $(\vb_i,-\va_i, \vec 0)$, $(\vec 0, \va_i, \ve_j)$, and $(\va_i + \vb_i,\va_i + \vb_i, \vec 0)$, which we again permute to be in columns $i$, $2r+j$, and $r+i$, respectively.  The result is a matrix we will call $M_2$, which is a block matrix of the form
\[M_2 = \left[ 
\begin{array}{c|c} 
  M_3 & B \\ 
  \hline 
0 & I_k 
\end{array} 
\right], \]
where $M_3$ is a $(2r \X 2r)$-matrix.

Note that the final $k$ columns of $M_2$ are linearly independent, due to the $I_k$ block in the bottom right, and none of the first $2r$ columns are in the span of these columns; thus, $\text{rk}(M_2) = \text{rk}(M_3) + k$, which implies that $\nullity(M_3) = \nullity(M_2)$.  By construction, $M_3$ is a $2r \X 2r$-matrix such that column one is $(\va_1,\va_1) = (\ve_1,\ve_1)$, column $r+1$ is $(\vb_1,\vb_1)$, columns two through $r$ are of the form $(\va_i,-\vb_i)$, $(\vb_i,-\va_i)$, or $(\va_i + \vb_i, \vec 0)$ (corresponding to strands $s_i$ that meet no disk $\Delta_j$), and columns $r+2$ through $2r$ are of the form $(\va_i + \vb_i,\va_i + \vb_i)$ or $(\vec 0, \va_i + \vb_i)$ (corresponding to strands $s_i$ that meet no disk $\Delta_j)$.

We perform column operations on $M_3$:  First, we replace the pair $(\va_1,\va_1)$ and $(\vb_1,\vb_1)$ with $(\va_1,\va_1)$ and $(\va_1 + \vb_1,\va_1 + \vb_1)$.  Next, we replace the pair $(\va_i + \vb_i, \vec 0)$ and $(\vec 0, \va_i + \vb_i)$ with $(\va_i + \vb_i, \vec 0)$ and $(\va_i + \vb_i,\va_i + \vb_i)$.  Finally, we exchange column $i$ with column $i+r$, where $1 \leq i \leq r$.  The result, a matrix we call $M_4$, has columns 1 through $r$ of the form $(\va_i + \vb_i,\va_i + \vb_i) = (\vs_i,\vs_i)$, column $r+1$ of the form $(\va_1,\va_1) = (\ve_1,\ve_1)$, and columns $r+2$ through $2r$ of the form $(\va_i,-\vb_i)$, $(\vb_i,-\va_i)$, or $(\va_i + \vb_i, \vec 0)$.  In other words, $M_4$ can be expressed as a block matrix
\[M_4 = \left[ 
\begin{array}{c|c} 
  M_D & C \\ 
  \hline 
M_D & D 
\end{array} 
\right]. \]
Now, we perform row operations on $M_4$.  First, subtract row $i$ from row $r+i$, so that the resulting columns are of the forms $(\vs_i,\vec 0)$, $(\va_1,\vec 0)$, and $(\va_i,-\va_i-\vb_i)$, $(\vb_i,-\va_i-\vb_i)$, or $(\va_i + \vb_i,-\va_i-\vb_i)$.  Next, scale rows $r+1$ through $2r$ by a factor of $-1$, and then add row $r+i$ to row $i$ for $1\leq i \leq r$.  The resulting matrix, call it $M_5$, has columns 1 through $r$ of the form $(\vs_i,\vec 0)$, column $r+1$ of the form $(\ve_1,\vec 0)$, and columns $r+2$ through $2r$ of the form $(-\vb_i,\va_i + \vb_i) = (-\vb_i,\vs_i)$, $(-\va_i,\va_i + \vb_i) = (-\va_i,\vs_i)$, or $(\vec 0, \va_i+\vb_i) = (\vec 0 , \vs_i)$.  In other words, $M_5$ can be expressed as a block matrix
\[M_5 = \left[ 
\begin{array}{c|c} 
  M_D & A \\ 
  \hline 
0 & M_D' 
\end{array} 
\right]. \]
By construction, the first column of $A$ is $\ve_1$, and $M_D'$ is obtained from $M_D$ by replacing its first column with the zero vector, completing the proof.
\end{proof}

\begin{remark}\label{rmk:sum}
If we carry out the proof of Proposition~\ref{prop:matrix} with $J \# -\!J$, in which the only tangle replacement is carried out in $\Delta_0$, we see that $J \# -\!J$ has a coloring matrix with the same mod $p$ nullity as
\[M ' = \left[ 
\begin{array}{c|c} 
  M_D & A' \\ 
  \hline 
0 & M_D' 
\end{array} 
\right],\]
where $D$ is a diagram for $J$, and the only nonzero entry in the matrix $A'$ is $a_{11} = 1$.  By Proposition~\ref{prop:sum}, we have that $\col_p(J \# -\!J) = \frac{(\col_p(J))^2}{p}$, which implies
\[\nullity_p(M') = 2 \cdot \nullity_p(M_D) - 1.\]
\end{remark}

We have now assembled the ingredients we need to prove our main theorem.

\begin{proof}[Proof of Theorem~\ref{thm:main}]
Suppose that $K$ admits a symmetric union with partial knot $J$.  By Proposition~\ref{prop:matrix}, there is an $r$-crossing diagram $D$ for $J$ such that any coloring matrix for $K$ has the same mod $p$ nullity as a $2r \times 2r$ block matrix $M$ of the form
\[ M = \left[ 
\begin{array}{c|c} 
  M_D & A \\ 
  \hline 
0 & M_D' 
\end{array} 
\right],\]
in which $A$ is a matrix with $\ve_1$ as its first column, and $M_D'$ is obtained from $M_D$ by replacing its first column with the zero vector.  Note that
\[ \col_p(K) = (\nullity_p (M))^p.\]
Since the lower left block of $M$ is the zero block, we can use a basis for the null space of $M_D$ to construct a linearly independent set of vectors of the same size in the null space of $M$; thus
\begin{equation}\label{eq1}
\nullity_p(M_D) \leq \nullity_p(M).
\end{equation}
\noindent (This inequality corresponds to the fact that any $p$-coloring of a partial knot $J$ induces a $p$-coloring of $K$, as noted in Section~\ref{sec:prelim}.)

On the other hand, by Remark~\ref{rmk:sum}, the matrix
\[M ' = \left[ 
\begin{array}{c|c} 
  M_D & A' \\ 
  \hline 
0 & M_D' 
\end{array} 
\right],\]
where $A'$ has only one non-zero entry $a_{11}=1$, satisfies $\nullity(M') = 2 \cdot \nullity(M_D) - 1$.  Due to the block structure of $M'$ and the zero block in the lower left, a sequence of row operations converting $M'$ to row echelon form will only combine the first $r$ rows with each other and the second $r$ rows with each other.  As such, the same row operations performed on $M$ will result in a pivot element in every column corresponding to a pivot in the row echelon form of $M'$.  It follows that $\text{rk}(M) \geq \text{rk}(M')$, and so the rank-nullity theorem implies
\begin{equation}\label{eq2}
\nullity_p(M) \leq \nullity_p(M') = 2 \cdot \nullity_p(M_D) - 1.
\end{equation}
Combining Equations (\ref{eq1}) and (\ref{eq2}) yields
\[ \nullity_p(M_D) \leq \nullity_p(M) \leq 2 \cdot \nullity_p(M_D) - 1,\]
or equivalently,
\[ \col_p(J) \leq \col_p(K) \leq \frac{(\col_p(J))^2}{p},\]
completing the proof.
\end{proof}

\begin{remark}\label{rmk:range}
We note that either inequality from Theorem~\ref{thm:main} can be equality.  Consider the examples shown in Figure~\ref{fig:steve}.  In this case, the partial knot $J$ is the trefoil, with $\col_3(J) = 9$, and two different ribbon knots $K = 6_1$ and $K' = 9_{46}$ are constructed with partial knot $J$.  We leave it to the reader to compute that $\col_3(K) = 9$ while $\col_3(K') = 27$, realizing the lower and upper bounds from Theorem~\ref{thm:main}.
\end{remark}

\begin{proof}[Proof of Corollary~\ref{cor:main}]
For any odd prime $p$, let $T_p$ denote the $(p,2)$-torus knot.  It is well-known that $\det(T_p) = p$ and $T_p$ is 2-bridge, and thus Lemma~\ref{lem:detcol} and Lemma~\ref{lem:bridge} imply $\col_p(T_p) = p^2$.  Define $K_p = T_p \# T_p \# T_p \# T_p$.  Then by repeated applications of Proposition~\ref{prop:sum}, we have $\col_p(K_p) = p^5$, and since knot determinant is multiplicative under connected sum, $\det(K_p) = p^4$.  On the other hand, let $J_p$ denote the $(p^4,2)$-torus knot.  Then $J_p$ is 2-bridge with $\det(J_p) = p^4$, and so again we have $\col_p(J_p) = p^2$ by applying Lemmas~\ref{lem:detcol} and~\ref{lem:bridge}.  While $\det(K_p) = \det(J_p)$, it follows immediately from Theorem~\ref{thm:main} that $K_p$ and $J_p$ are not symmetrically related, since any knot $K$ with a symmetric union presentation with partial knot $J_p$ satisfies $\col_p(K) \leq \frac{(\col_p(J_p))^2}{p} = p^3$, while any knot $K'$ with symmetric union presentation with partial knot $K_p$ satisfies $\col_p(K') \geq \col_p(K_p) = p^5$.

Now, let $p_1,\dots,p_k$ be $k$ distinct odd primes, and for any tuple $\mathbf{x} = (x_1,\dots,x_m) \in \Z_2^m$, define the knot $K_{\mathbf x}$ by 
\[ K_{\mathbf x} = \#_{i=1}^m C_i,\]
where $C_i = J_{p_i}$ if $x_i = 0$ and $C_i = K_{p_i}$ if $x_i = 1$.  Since $\det(J_{p_i}) = \det(K_{p_i}) = p_i^4$, we have that for any $\mathbf x$,
\[ \det(K_{\mathbf x}) = p_1^4 \cdot \dots \cdot p_m^4.\]
Moreover, note that if $j \neq i$, then $p_i$ is not a factor of $\det(C_j) = p_j^4$, and so $C_j$ is not $p_i$-colorable by Lemma~\ref{lem:detcol}.  It follows from Corollary~\ref{cor:notcolor} that
\begin{equation}\label{eq3}
\col_{p_i}(K_{\mathbf x}) = \col_{p_i}(C_i).
\end{equation}
Finally, suppose that $\mathbf x' \neq \mathbf x$.  Then there is some index $i$ at which they disagree, and we can suppose without loss of generality that the summand $C_i$ of $K_{\mathbf x}$ satisfies $C_i= J_{p_i}$, while the corresponding summand $C_i'$ of $K_{\mathbf x'}$ satisfies $C_i' = K_{p_i}$.  By Equation~\ref{eq3}, we have $\col_{p_i}(K_{\mathbf x}) = \col_{p_i}(J_{p_i}) = p_i^2$, while $\col_{p_i}(K_{\mathbf x'}) = \col_{p_i}(K_{p_i}) = p_i^5$, and by an argument identical to the one for $J_p$ and $K_p$ above, we conclude that $K_{\mathbf x}$ and $K_{\mathbf x'}$ are not symmetrically related.  As $|\Z_2^m| = 2^m$, this completes the proof.
\end{proof}

\begin{remark}
We can also use Theorem~\ref{thm:main} to find knots prime knots in the knot table~\cite{knotinfo} with the same determinant that are not symmetrically related.  As one example, both $8_{10}$ and $12n_{642}$ have determinant $27$, but we can compute $\col_3(8_{10}) = 9$, while $\col_3(12n_{642}) = 81$.  It follows that if $K$ has $8_{10}$ as a partial knot, then $\col_3(K) \leq 27$, where if $K'$ has $12n_{642}$ as a partial knot, then $\col_3(K') \geq 81$, and so $8_{10}$ and $12n_{642}$ are not symmetrically related.
\end{remark}

\section{Questions}\label{sec:quest}

We conclude with questions to stimulate research in this area.  One natural question is to push this study further to find additional obstructions to two knots being symmetrically related.

\begin{question}
Do there exist knots $J$ and $J'$ such that $\det(J) = \det(J')$ and $\col_p(J) = \col_p(J')$ for all $p$ but such that $J$ and $J'$ are not symmetrically related?
\end{question}

Although the terminology suggests that being symmetrically related is an equivalence relation, transitivity is not immediately clear.

\begin{question}
Suppose $J$ is symmetrically related to $J'$ and $J'$ is symmetrically related to $J''$.  Is $J$ necessarily symmetrically related to $J''$?
\end{question}

We suspect that symmetrically related knots are somewhat rare, and so we pose the following.

\begin{question}
Does there exist a knot $J$ that is symmetrically related to infinitely many other knots?  Or is every knot $J$ symmetrically related to only finitely many knots?
\end{question}

In light of Proposition~\ref{prop:notsymm}, we know that there exist symmetric union presentations with $p$-colorings that are not symmetrically compatible.  However, this example corresponds to the knot $3_1 \# -\!3_1$, which has a different and more natural symmetric union presentation in which all $p$-colorings are indeed symmetrically compatible.  This suggests

\begin{question}\label{qsymm}
If $K$ is a symmetric ribbon knot, does there exist a symmetric union presentation for $K$ such that all $p$-colorings are symmetrically compatible?
\end{question}

An affirmative answer to Question~\ref{qsymm} could provide an additional tool to further our understanding of Question~\ref{q1}, truly the main problem underlying the study of symmetric unions.

\bibliographystyle{amsalpha}
\bibliography{partialknots}

\end{document}